\documentclass[11pt]{article}

\usepackage[T1]{fontenc}
\usepackage{newtxtext,newtxmath}
\usepackage{microtype}
\usepackage{geometry}
\usepackage{setspace}
\usepackage{graphicx}
\usepackage{caption}
\usepackage{mathtools}
\usepackage{bm}
\usepackage{cite}

\usepackage{amsthm}
\usepackage{hyperref}
\hypersetup{
    colorlinks=true,
    linkcolor=blue,
    urlcolor=blue,
    citecolor=blue,
    pdfborder={0 0 0}
}
\makeatletter
\def\Hy@raisedlink#1{#1}  
\makeatother

\theoremstyle{plain}
\newtheorem{theorem}{Theorem}
\newtheorem{lemma}{Lemma}
\newtheorem{proposition}{Proposition}

\theoremstyle{definition}
\newtheorem{definition}{Definition}
\newtheorem{construction}{Construction}
\newtheorem{remark}{Remark}

\title{\Large\bfseries
M\"obius Transformations and the Analytic--Geometric Reconstruction of the Induction–Machine Circle Diagram
}

\let\oldeqref\eqref
\renewcommand{\eqref}[1]{Eq.~\oldeqref{#1}}

\author{
\large Anubhav Gupta\thanks{Email: \texttt{anubhav.gupta@inorbitaerospace.com}}\;
\href{https://orcid.org/0000-0002-3216-868X}{\includegraphics[height=10pt]{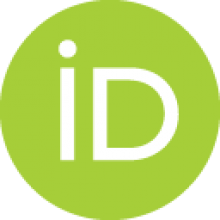}}\\[4pt]
\small Guidance, Navigation, and Control (GN\&C) Engineer\\
\small In Orbit Aerospace Inc., Torrance, CA 90501, USA\\
\small Visiting Researcher, Department of Aerospace Engineering Sciences\\
\small University of Colorado Boulder, CO 80303, USA
}

\date{}

\begin{document}
\maketitle
\vspace{-0.75em}

\begin{abstract}
The Heyland circle diagram is a classical graphical method for representing the steady--state behavior of induction machines using no--load and blocked--rotor test data. Despite its long pedagogical history, the traditional geometric construction has not been formalized within a closed analytic framework. This note develops a complete Euclidean reconstruction of the diagram using only the two measured phasors and elementary geometric operations, yielding a unique circle, a torque chord, a slip scale, and a maximum--torque point. We prove that this constructed circle coincides precisely with the analytic steady--state current locus obtained from the per--phase equivalent circuit. A M\"obius transformation interpretation reveals the complex--analytic origin of the diagram’s circularity and offers a compact explanation of its geometric structure.
\end{abstract}


\section{Introduction}
The Heyland circle diagram is a century--old graphical tool for summarizing the steady--state behavior of induction machines using only the no--load and blocked--rotor tests. Two measured current phasors are used to construct a circle in the complex current plane; auxiliary lines then encode power factor, torque, slip, and efficiency. Classical expositions describe the diagram via geometric steps---drawing phasors, erecting perpendiculars, and constructing diameter lines---but do not formalize the geometry in analytic 
terms \cite{theraja2014textbook, kothari2011electric}.

Modern texts either omit the diagram or focus on equivalent--circuit 
analysis, where the steady--state currents are obtained by algebraic 
manipulation rather than geometric construction \cite{krause2013analysis}. Numerical and algorithmic Cartesian reconstructions have also been described in the literature (e.g., \cite{ghosh2022matlab}), where slopes, intersections, and coordinate calculations are used to generate the diagram step by step from test data. Such approaches recover the circle computationally but do not isolate the underlying Euclidean geometry or establish its analytic properties.

This note presents a formal analytic--geometric reconstruction of the 
classical Heyland circle \cite{heyland1906graphical} using only the two measured phasors. The procedure is stated precisely, and its geometric and analytic properties are established. In particular, we show that
\begin{itemize}
    \item the circle obtained from the perpendicular--bisector geometry is unique under the classical reference--horizontal assumption;
    \item the torque chord and maximum--torque point arise naturally from an orthogonality condition;
    \item the reconstructed circle coincides with the analytic input--current locus derived from the per--phase equivalent circuit.
\end{itemize}

A numerical reconstruction of the diagram using spreadsheets and plotting software was previously presented by the author in \cite{gupta2012testing}, and algorithmic coordinate methods appear in works such as \cite{ghosh2022matlab}. The contribution of the present note is to give a concise analytic--geometric formulation with explicit existence and uniqueness results together with a formal equivalence to the analytic current locus.

To the author's knowledge, a M\"obius transformation interpretation of the Heyland diagram and an explicit equivalence between the geometric construction and the analytic current locus have not appeared in the literature.

\section{Geometric Setup and Classical Reconstruction}
\label{sec:construction}
We interpret phasors as complex numbers embedded in $\mathbb{R}^2$. Let
\[
I = I_m(\cos\phi + j\sin\phi)
\quad\longleftrightarrow\quad
(x,y) = (I_m\cos\phi,\, I_m\sin\phi).
\]

\begin{definition}[Test points]
Let $p_0 = (x_0,y_0)$ denote the endpoint of the no--load current phasor $\,I_0\angle\phi_0\,$ and let $p_A = (x_A,y_A)$ denote the endpoint of the blocked--rotor current phasor $\,I_{sc}\angle\phi_{sc}\,$ referred to rated voltage. Assume $p_0 \neq p_A$.
\end{definition}

\begin{definition}[Output line]
The \emph{output line} is the line $L_{O'A}$ through $p_0$ and $p_A$.
\end{definition}

\begin{construction}[Analytic--Geometric Reconstruction]
\label{constr:recon}
Given test points $p_0, p_A \in \mathbb{R}^2$:

\begin{enumerate}
    \item \textbf{Midpoint.} Define the midpoint of output line $L_{O'A}$
    \[
    C' = \frac{1}{2}(p_0 + p_A).
    \]

    \item \textbf{Circle center.} Let $m_{O'A}$ be the slope of $L_{O'A}$; if 
    $L_{O'A}$ is vertical, interpret $m_{O'A} = \infty$. The line through $C'$ 
    with slope $-1/m_{O'A}$ (interpreted as vertical if $L_{O'A}$ is horizontal) intersects the horizontal line $y = y_0$ at a unique point $C = (x_C,y_0)$. Define
    \[
    r = \|C - p_0\|.
    \]
    The circle diagram is the circle with center $C$ and radius $r$.

    \item \textbf{Torque chord.} Define the point
    \[
    E = \bigg(x_A,\; \frac{1}{2}(y_0 + y_A)\bigg).
    \]
    The segment joining $p_0$ and $E$ is the \emph{torque chord}.

    \item \textbf{Maximum--output point.} The line through $C$ and $C'$ with slope $-1/m_{O'A}$ intersects the circle diagram at the \emph{maximum--output point} $M_O$.

    \item \textbf{Maximum--torque point.} Let $m_{O'E}$ be the slope of the torque chord. The line through $C$ with slope $-1/m_{O'E}$ intersects the circle diagram at the \emph{maximum--torque point} $M_T$. 

    \item \textbf{Slip scale and efficiency line.} A slip scale is obtained by drawing a line parallel to the torque chord and calibrating 0\% and 100\% slip via vertical projections. An efficiency line is obtained by extending $L_{O'A}$ to meet the horizontal through the top of the circle and projecting radial lines from the origin.
\end{enumerate}
\end{construction}

The following results show that this reconstruction is well posed and recovers the classical geometry.

\subsection*{Uniqueness of the Circle}

\begin{proposition}[Uniqueness given a reference horizontal]
\label{prop:uniqueness}
Let $p_0, p_A \in \mathbb{R}^2$ with $p_0 \neq p_A$ and fix the horizontal line $y = y_0$ through $p_0$. Then Construction~\ref{constr:recon} yields a unique circle that passes through both $p_0$ and $p_A$ whose center lies on $y = y_0$.
\end{proposition}

\begin{proof}
The midpoint $C'$ of the segment $p_0p_A$ is uniquely defined. The line through $C'$ perpendicular to $L_{O'A}$ is uniquely determined and intersects the horizontal line $y = y_0$ in exactly one point $C = (x_C,y_0)$. Because $C'$ is the midpoint of $p_0$ and $p_A$, any point on the perpendicular through $C'$ satisfies $\|C - p_0\| = \|C - p_A\|$. In particular, the constructed $C$ has equal distance to $p_0$ and $p_A$.

Conversely, suppose $\tilde{C}$ is a point on $y = y_0$ such that the circle centered at $\tilde{C}$ passes through both $p_0$ and $p_A$. Then
\[
\|\tilde{C} - p_0\| = \|\tilde{C} - p_A\|.
\]
The locus of such points is the perpendicular bisector of $p_0p_A$, i.e., the line through $C'$ orthogonal to $L_{O'A}$. Its intersection with $y = y_0$ is unique, so $\tilde{C} = C$. Thus, exactly one circle with center on $y = y_0$ passes through both $p_0$ and $p_A$.
\end{proof}

\subsection*{Torque Chord and Maximum Torque}
\begin{lemma}[Orthogonality and maximum torque]
\label{lem:MT}
Let $L_{O'E}$ be the torque chord defined in 
Construction~\ref{constr:recon}, and let $L_\perp$ be the line through $C$ orthogonal to $L_{O'E}$. Then $L_\perp$ intersects the circle diagram in a unique point $M_T$ in the physically admissible region, and this point corresponds to the classical maximum--torque operating condition.
\end{lemma}

\begin{proof}
Because the circle diagram is a compact, strictly convex curve, the line $L_\perp$ passing through the center $C$ intersects it in two antipodal points. Only one of these lies in the region conventionally used to represent motoring operation; denote it by $M_T$.

Classical circle--diagram analysis (see, e.g., 
\cite{langsdorf1937theory, fitzgerald2002electric}) shows that torque is proportional to the air--gap power, which in turn corresponds to a specific geometrically represented component of the current. The torque chord $L_{O'E}$ is drawn so that projections of operating points onto this chord encode that component. An extremum of torque with respect to slip occurs when the derivative of this projection vanishes. Geometrically, this happens when the radius vector at the operating point is orthogonal to the torque chord, i.e., when the operating point lies on $L_\perp$. Thus $M_T$ represents the maximum--torque operating condition.
\end{proof}

\begin{remark}
Lemma~\ref{lem:MT} recovers the classical maximum--torque construction in purely geometric terms, without invoking explicit torque formulas.
\end{remark}

\section{Equivalence to the Analytic Current Locus}
\label{sec:equivalence}
We now connect the circle diagram obtained from the analytic--geometric reconstruction to the steady--state current locus predicted by the per--phase equivalent circuit. The complete algebraic derivation is given in Appendix~\ref{app:analytic}.

\begin{theorem}[Equivalence to the equivalent--circuit locus]
\label{thm:equivalence}
Under the standard assumptions of balanced, sinusoidal steady--state 
operation with slip--independent parameters, the circle diagram obtained from Construction~\ref{constr:recon} coincides with the input--current locus of the induction machine derived from the per--phase equivalent circuit.
\end{theorem}

\begin{proof}
Classical equivalent--circuit analysis represents the stator and 
magnetizing branches by impedances $Z_s$ and $Z_m$, forms their 
Th\'evenin equivalent as seen from the rotor branch, and expresses the rotor current $I_r(s)$ as a function of slip $s$. In 
Appendix~\ref{app:analytic} we show that the real and imaginary parts of either the rotor current $I_r(s)$ or the input current $I(s)$ satisfy an equation of the form
\[
(x - x_c)^2 + (y - y_c)^2 = r^2,
\]
so that the corresponding current locus is a circle with center $(x_c,y_c)$ and radius $r$ determined by the equivalent--circuit parameters.

The no--load and blocked--rotor operating points correspond to specific values of slip and therefore lie on this analytic circle. Proposition \ref{prop:uniqueness} implies that, once the reference horizontal $y = y_0$ through $p_0$ is fixed, there exists a unique circle centered on this line and passing through both test points. Appendix~\ref{app:analytic} further shows that the center and radius computed from the equivalent circuit coincide with those obtained via the perpendicular--bisector geometry. Hence the circle diagram produced by Construction~\ref{constr:recon} and the analytic equivalent--circuit current locus are identical.
\end{proof}

\begin{remark}
Theorem~\ref{thm:equivalence} shows that the classical circle diagram is precisely consistent with the steady--state equivalent--circuit model under its usual assumptions, rather than being an approximate or heuristic representation.
\end{remark}

\section{M\"obius Transformations and the Current Locus}
\label{sec:mobius}
The rotor current $I_r(s)$ obtained from the per--phase equivalent circuit admits a natural formulation in the language of complex analysis. After introducing the variable $z = 1/s$, the mapping $s \mapsto I_r(s)$ becomes a fractional linear (Möbius) transformation \cite{ferdinand1827barycentrische, needham2023visual}. Since Möbius maps send lines to circles  (see, e.g., \cite{needham2023visual}), this provides a succinct analytic explanation for the circular nature of the Heyland diagram.

\subsection{Rotor Current as a Möbius Map}
Let
\[
    z = \frac{1}{s}, \qquad 
    Z_r(s) = \frac{R_r'}{s} + jX_r' = R_r' z + jX_r',
\]
so that the rotor branch impedance becomes affine in $z$.
Using the Th\'evenin equivalent seen from the rotor branch,
\[
    Z_{\mathrm{th}} = R_{\mathrm{th}} + jX_{\mathrm{th}}, \qquad
    V_{\mathrm{th}} = V_{\mathrm{th},x} + j V_{\mathrm{th},y},
\]
the rotor current becomes
\begin{equation}
\label{eq:mobius-basic}
    I_r(z)
    = \frac{V_{\mathrm{th}}}
           {R_{\mathrm{th}} + R_r' z + j(X_{\mathrm{th}} + X_r')}.
\end{equation}

Define the complex constants
\[
    \alpha = 0,\qquad
    \beta  = V_{\mathrm{th}},\qquad
    \gamma = R_r',\qquad
    \delta = R_{\mathrm{th}} + j(X_{\mathrm{th}} + X_r').
\]
Then \eqref{eq:mobius-basic} takes the standard fractional linear form
\begin{equation}
\label{eq:mobius-form}
    I_r(z)
    = \frac{\alpha z + \beta}{\gamma z + \delta}
\end{equation}

\begin{proposition}
\label{prop:mobius-clean}
The mapping $z \mapsto I_r(z)$ defined by \eqref{eq:mobius-form} is a nondegenerate Möbius transformation on the extended complex plane $\widehat{\mathbb{C}}$.
\end{proposition}

\begin{proof}
A Möbius transformation has the form $f(z) = (az+b)/(cz+d)$ with 
$ad-bc \neq 0$. In \eqref{eq:mobius-form}, $a=\alpha=0$, $b=\beta=V_{\mathrm{th}}$, $c=\gamma=R_r'$, and $d=\delta$. The determinant is
\[
    ad - bc = -V_{\mathrm{th}}R_r' \neq 0
\]
under the usual physical assumptions $R_r' > 0$ and $V_{\mathrm{th}} \neq 0$. Thus, \eqref{eq:mobius-form} defines a Möbius transformation.
\end{proof}

\paragraph{Geometric interpretation.}
During steady--state operation the variable $z = 1/s$ lies on the real axis, so $I_r(z)$ maps a line through a Möbius transformation. Since Möbius maps send lines to circles, the image $I_r(\Gamma^\star)$ must be circular. The no--load and blocked--rotor phasors correspond to distinct real values of $z$ and lie on this circle; by Proposition~\ref{prop:uniqueness}, it is the same circle obtained from the geometric reconstruction.

\subsection{Circularity of the Rotor--Current Locus}
We begin with a classical structural property of Möbius transformations.

\begin{theorem}[Lines and circles are preserved]
\label{thm:mobius-lines}
A Möbius transformation maps every line or circle in $\mathbb{C}$ to 
another line or circle in $\mathbb{C}$.
\end{theorem}

In steady--state operation the slip $s$ is real, and the substitution 
$z = 1/s$ therefore restricts the physically admissible domain of the 
mapping $I_r(z)$ to the punctured real axis,
\[
    \Gamma^\star = \{ z \in \mathbb{R} : z \neq 0 \}.
\]
By Proposition~\ref{prop:mobius-clean}, the function $I_r$ is a 
nondegenerate Möbius transformation; hence, by Theorem~\ref{thm:mobius-lines}, the image $I_r(\Gamma^\star)$ must lie on either a straight line or a circle. Since $I_r$ is bounded on $\Gamma^\star$ and the transformation is not degenerate, the image cannot be a line and must therefore be a circle.

\begin{proposition}[Analytic origin of the Heyland circle]
\label{prop:analytic-circle}
Under steady--state conditions with slip--independent parameters, the 
rotor--current locus $I_r(\Gamma^\star)$ is a circle in the complex 
plane. The no--load and blocked--rotor phasors lie on this circle, whose center and radius agree with those obtained via the geometric reconstruction of Section~\ref{sec:construction}.
\end{proposition}

\begin{proof}
As $z \to +\infty$ (equivalently, $s \to 0^{+}$), the Möbius 
representation~\eqref{eq:mobius-form} yields
\[
    I_r(+\infty)
    = \lim_{z\to+\infty} \frac{V_{\mathrm{th}}}{\gamma z + \delta}
    = 0,
\]
so the no--load current $I_0$ lies on the image circle. The blocked--rotor test corresponds to a finite real value of $z$ and thus also produces a point of $I_r(\Gamma^\star)$.

Because a nondegenerate Möbius transformation maps the real axis to a circle, and the no--load and blocked--rotor phasors lie on this image, Proposition~\ref{prop:uniqueness} implies that this circle is the unique one passing through $p_0$ and $p_A$ whose center lies on the prescribed horizontal reference. This coincides with the circle obtained from the analytic--geometric construction.
\end{proof}

\begin{remark}[Generating operation: negative slip]
For $s<0$ (generating mode), the point $z = 1/s$ lies on the negative real axis. The domain remains a subset of a line, so its image under the Möbius map lies on the same circle. Negative slip corresponds to points on the opposite side of the diagram, consistent with classical representations of generating operation. Thus, the Möbius interpretation provides a unified analytic description of motoring and generating modes.
\end{remark}

\subsection{Geometric Meaning of the Maximum--Torque Condition}
For the specialized Möbius form
\[
    I_r(z)
    = \frac{V_{\mathrm{th}}}{\gamma z + \delta}
    = \frac{\beta}{\gamma z + \delta},
\]
the derivative with respect to $z$ is
\[
    \frac{dI_r}{dz}
    = -\frac{\beta \gamma}{(\gamma z + \delta)^{2}}.
\]
Since $s = 1/z$, the chain rule gives
\[
    \frac{dI_r}{ds}
    = \frac{dI_r}{dz}\,\frac{dz}{ds}
    = -\frac{\beta \gamma}{(\gamma z + \delta)^2}\!
      \left(-\frac{1}{s^{2}}\right)
    = \frac{\beta \gamma}{(\gamma z + \delta)^{2} s^{2}}.
\]

In classical complex–power theory, torque is proportional to  
$\Re\{V_{\mathrm{th}}^{*} I_r\}$. Multiplying the current by the fixed complex constant $V_{\mathrm{th}}^{*}$ simply rotates and scales the complex plane. After absorbing this rotation into the reference frame, the torque can be written (up to a constant factor) as the imaginary part of the rotated current phasor. Thus, without loss of generality,
\[
    T(s) \ \propto\ \Im\!\{\, I_r(s)\, V_{\mathrm{th}}^{*} \}.
\]
This formulation is equivalent to the standard real–power expression but is better suited to M\"obius analysis from an algebraic standpoint. Hence the extremum condition $dT/ds = 0$ becomes
\[
    \Im\!\left\{
        \frac{dI_r}{ds}\, V_{\mathrm{th}}^{*}
    \right\} = 0.
\]
Substituting the expression for $dI_r/ds$ yields
\[
    \Im\!\left\{
        \frac{\beta \gamma\, V_{\mathrm{th}}^{*}}
             {(\gamma z + \delta)^{2}\, s^{2}}
    \right\} = 0.
\]
The factor $\beta \gamma V_{\mathrm{th}}^{*}$ is a nonzero complex constant and therefore irrelevant to the phase constraint. Thus, the extremum condition reduces to
\[
    \arg\!\big( (\gamma z + \delta)^{-2} \big)
    \in \{0,\pi\}.
\]

Geometrically, the tangent to the current locus is proportional to 
$dI_r/ds$. Requiring the above imaginary part to vanish is equivalent to demanding that this tangent be parallel to the classical torque line. Since the tangent to a circle is everywhere perpendicular to the radius, this is equivalent to requiring that the radius vector at the operating point be orthogonal to the torque chord:
\[
    \arg(I_r - C)
    \ =\
    \arg(\text{normal to torque chord}).
\]
This is precisely the geometric condition obtained in Lemma~\ref{lem:MT}, establishing the equivalence of the analytic M\"obius--phase criterion and the geometric orthogonality condition.

The analytic and complex--analytic characterizations developed above both predict a circular current locus passing through the no--load and blocked--rotor points. The following example illustrates these results using measured test data and displays the agreement between the analytic locus, the Möbius image of the slip axis, and the geometric reconstruction of Section~\ref{sec:construction}.

\section{Example: M\"obius and Analytic Reconstruction of the Circle Diagram}
\label{sec:example}

\begin{figure}[ht!]
    \centering
    \includegraphics[width=0.95\linewidth]{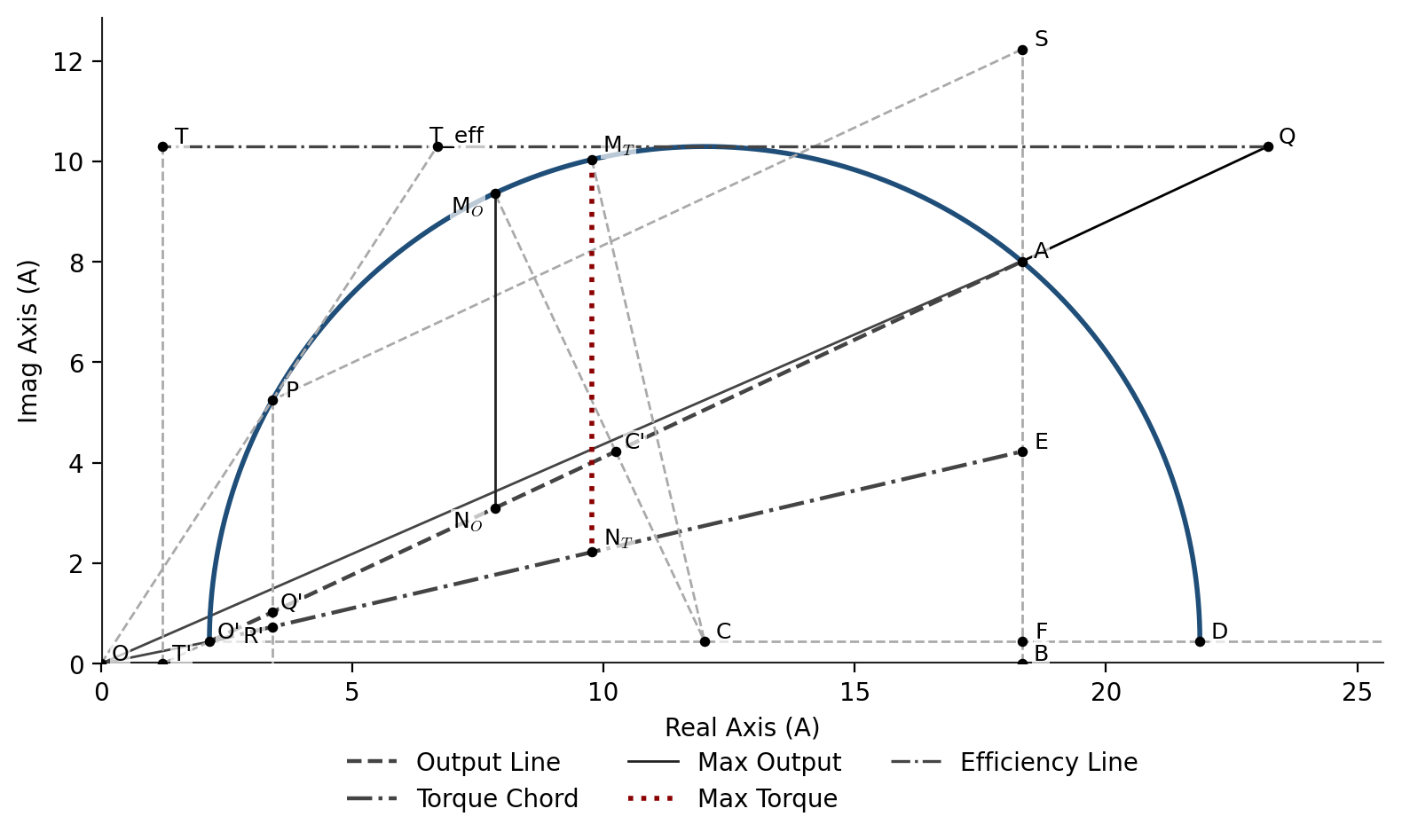}
    \caption{Image of the slip axis under the Möbius transformation
    \(I_r(z)=V_{\mathrm{th}}/(\gamma z+\delta)\). Positive real values of
    \(z\) (\(z>0\)) correspond to motoring operation (\(s>0\)), while
    negative values (\(z<0\)) correspond to generating operation (\(s<0\)).
    Both branches map to the same circular current locus (blue arc). The
    no–load point \(O'\) and blocked–rotor point \(A\) lie on this locus
    and uniquely determine its center \(C\) once the reference horizontal
    through \(O'\) is fixed. Also shown are the classical geometric
    constructions: the output line (dashed), torque chord (dash–dot),
    maximum–output line (solid), maximum–torque line (red dotted), and the
    efficiency line (grey).}

    \label{fig:diagram}
\end{figure}

To illustrate the results of Sections~\ref{sec:construction}--\ref{sec:mobius}, we apply both the analytic locus derived from the equivalent circuit and the Möbius representation of \eqref{eq:mobius-form} to a representative test case. The mapping $z = 1/s$ sends steady--state slip values to the real axis, and the resulting image under the Möbius transformation $I_r(z)$ produces a circular current locus, as established in Proposition~\ref{prop:analytic-circle}. Independently, the analytic expressions for the real and imaginary components of the current yield the same circle, consistent with Theorem~\ref{thm:equivalence}.

The example in Fig.\ref{fig:diagram} illustrates this agreement visually. The no--load and blocked--rotor phasors, computed from the test data of \cite{theraja2014textbook}, lie on the Möbius image of the slip axis and determine the same center and radius as those obtained from the analytic--geometric construction of Section~\ref{sec:construction}.  
The resulting diagram recovers the classical geometric features—torque chord, slip calibration, and maximum--torque point—while simultaneously displaying their complex--analytic origin.

The example confirms that all three perspectives—geometric, analytic, and M\"obius—produce the same operating diagram and recover the classical features of the Heyland construction.  We now summarize the main contributions and outline directions for generalization.

\section{Conclusion}
\label{sec:conclusion}
This work presented a concise analytic--geometric reconstruction of the classical Heyland circle diagram using only the no--load and blocked--rotor current phasors. The procedure produces a unique circle, torque chord, slip calibration, and maximum--torque point through elementary Euclidean constructions. We established existence and uniqueness of the reconstructed circle and showed that the corresponding operating features---including the maximum--torque condition---arise naturally from simple orthogonality relations.

The constructed diagram was further shown to coincide precisely with the steady--state current locus obtained from the per--phase equivalent circuit, thereby formalizing the long-assumed consistency between the geometric diagram and the analytic circuit model. A M\"obius transformation formulation of the rotor current provided a compact complex--analytic explanation for the circularity of the locus and unified the motoring and generating branches through the image of the real axis.

The framework developed here provides a foundation for computational generation and analysis of circle diagrams and clarifies the geometric structure underlying classical constructions. A companion work extends these methods to permanent-magnet synchronous
machines and spacecraft reaction wheels, where voltage constraints yield elliptical loci in the $dq$ plane and require a generalization of the analytic--geometric approach developed in this note.

\appendix
\section{Analytic Derivation of the Current Locus}
\label{app:analytic}
This appendix outlines the derivation of the circular rotor--current locus from the per--phase equivalent circuit and indicates how its center and radius coincide with those obtained from the geometric reconstruction in the main text.

Let the stator and magnetizing impedances be
\[
    Z_s = R_s + jX_s, 
    \qquad 
    Z_m = jX_m,
\]
and let the Th\'evenin equivalent seen from the rotor branch be
\[
    Z_{\mathrm{th}} 
    = \frac{Z_s Z_m}{Z_s + Z_m}
    = R_{\mathrm{th}} + jX_{\mathrm{th}}, 
    \qquad 
    V_{\mathrm{th}} 
    = V\,\frac{Z_m}{Z_s + Z_m}.
\]
The rotor branch referred to the stator is
\[
    Z_r(s) = \frac{R_r'}{s} + jX_r',
\]
so the rotor current is
\[
    I_r(s)
    = \frac{V_{\mathrm{th}}}
           {Z_{\mathrm{th}} + Z_r(s)}
    = \frac{V_{\mathrm{th}}}
           {R_{\mathrm{th}} + \frac{R_r'}{s}
            + j(X_{\mathrm{th}} + X_r')}.
\]

Define
\[
    a(s) = R_{\mathrm{th}} + \frac{R_r'}{s},
    \qquad 
    b = X_{\mathrm{th}} + X_r',
\]
so that
\[
    I_r(s)
    = \frac{V_{\mathrm{th}}}{a(s) + jb}
    = V_{\mathrm{th}}
      \,\frac{a(s) - jb}{a(s)^2 + b^2}.
\]
Let 
\[
    V_{\mathrm{th}} 
    = V_{\mathrm{th},x} + jV_{\mathrm{th},y}.
\]
Expanding the product gives
\[
    I_r(s)
    = \frac{1}{a(s)^2 + b^2}
      \bigl[
            (V_{\mathrm{th},x}a(s) + V_{\mathrm{th},y}b)
            + j\,(V_{\mathrm{th},y}a(s) - V_{\mathrm{th},x}b)
      \bigr],
\]
so that
\[
    x(s)
    = \Re\{ I_r(s) \}
    = \frac{V_{\mathrm{th},x} a(s) + V_{\mathrm{th},y} b}
           {a(s)^2 + b^2},
    \qquad
    y(s)
    = \Im\{ I_r(s) \}
    = \frac{V_{\mathrm{th},y} a(s) - V_{\mathrm{th},x} b}
           {a(s)^2 + b^2}.
\]

The dependence on $s$ enters only through $a(s)$, which appears linearly in the numerators and quadratically in the denominator. Solving the above expressions for $a(s)$ in terms of $(x,y)$ and eliminating $a(s)$ yields an implicit equation of the form
\[
    (x - x_c)^2 + (y - y_c)^2 = r^2,
\]
where $(x_c,y_c)$ and $r$ depend smoothly on 
$R_{\mathrm{th}}, X_{\mathrm{th}}, R_r', X_r'$, and $V_{\mathrm{th}}$.

Standard algebra (see, e.g., \cite{fitzgerald2002electric}) shows that the resulting center $(x_c,y_c)$ has the same coordinates as the point obtained by intersecting the perpendicular bisector of the segment joining the no--load and blocked--rotor points with the horizontal line through the no--load point. Likewise, the analytic radius coincides with the geometric distance from this center to either test point.

Thus the steady--state rotor--current locus derived from the equivalent circuit is a circle whose center and radius agree exactly with those computed by the geometric reconstruction.
\hfill$\square$

\bibliographystyle{ieeetr}
\bibliography{references}

@article{ferdinand1827barycentrische,
  title={Der barycentrische Calcul},
  author={Ferdinand, M{\"o}bius August},
  journal={Leipzig: Johan Ambrosius Barth},
  year={1827}
}

@book{fitzgerald2002electric,
  title={Electric Machinery},
  author={Fitzgerald, A.E. and Kingsley Jr., C. and Umans, S.D.},
  edition={6th},
  year={2002},
  publisher={McGraw-Hill}
}

@book{gupta2012testing,
  title={Testing of Transformers and Induction Machines},
  author={Gupta, Anubhav and Gupta, Abhinav},
  Author+an = {1=highlight},
  year={2012},
  address = {Charleston, SC},
  publisher={CreateSpace},
  edition={1st}
}

@article{ghosh2022matlab,
author = {Suchismita Ghosh},
title = {MATLAB Simulation of Circle Diagram in Three Phase Induction Motor},
journal = {IETE Journal of Education},
volume = {63},
number = {2},
pages = {63--77},
year = {2022},
publisher = {Taylor \& Francis},
doi = {10.1080/09747338.2022.2044395},
URL = {https://doi.org/10.1080/09747338.2022.2044395},
eprint = {https://doi.org/10.1080/09747338.2022.2044395}
}

@book{heyland1906graphical,
  title={A Graphical Treatment of the Induction Motor},
  author={Heyland, Alexander},
  year={1906},
  address = {New York, NY},
  publisher={McGraw Publishing Company}
}

@book{kothari2011electric,
  title={Electric Machines},
  author={Kothari, D.P. and Nagrath, I.J.},
  edition={4th},
  year={2011},
  publisher={McGraw-Hill}
}

@book{krause2013analysis,
  title={Analysis of Electric Machinery and Drive Systems},
  author={Krause, Paul C. and Wasynczuk, O. and Sudhoff, S.D. and Pekarek, S.D.},
  edition={3rd},
  year={2013},
  publisher={Wiley-IEEE Press}
}

@book{langsdorf1937theory,
  title={Theory of Alternating-Current Machinery},
  author={Langsdorf, Alexander Suss},
  year={1937},
  address = {New York: NY},
  publisher={McGraw-Hill Book Company}
}

@book{needham2023visual,
  title={Visual Complex Analysis},
  author={Needham, Tristan},
  year={2023},
  publisher={Oxford University Press}
}

@book{theraja2014textbook,
  title={A Textbook of Electrical Technology},
  author={Theraja, B. L. and Theraja, A. K.},
  year={2014},
  address = {New Delhi, India},
  publisher={S. Chand Publishing}
}

\end{document}